\documentclass[reqno, 12pt]{amsart}
\usepackage{amssymb}
\usepackage{amsfonts}
\usepackage{srcltx}
\usepackage{enumerate}
\usepackage{cite}

\newtheorem{theorem}{Theorem}[section]
\newtheorem{lemma}[theorem]{Lemma}

\newtheorem{corollary}[theorem]{Corollary}
\theoremstyle{definition}

\theoremstyle{remark}

\newtheorem*{acknowledgement}{Acknowledgement}

\DeclareMathOperator{\diam}{diam}
\DeclareMathOperator{\conv}{conv}

\DeclareMathOperator{\cl}{cl}

\begin{document}
\title[The fixed point property in direct sums]{The fixed point property in
direct sums \\
and modulus $R(a,X)$}
\author[A. Wi\'{s}nicki]{Andrzej Wi\'{s}nicki}

\begin{abstract}
We show that the direct sum $(X_{1}\oplus ...\oplus X_{r})_{\psi }$ with a
strictly monotone norm has the weak fixed point property for nonexpansive
mappings whenever $M(X_{i})>1$ for each $i$ $=1,...,r.$ In particular, $%
(X_{1}\oplus ...\oplus X_{r})_{\psi }$ enjoys the fixed point property if
Banach spaces $X_{i}$ are uniformly nonsquare. This combined with the
earlier results gives a definitive answer for $r=2$: the direct sum\ $%
X_{1}\oplus _{\psi }X_{2}$ of uniformly nonsquare spaces with any monotone
norm has FPP. Our results are extended for asymptotically nonexpansive
mappings in the intermediate sense.
\end{abstract}

\subjclass[2010]{Primary: 47H10; Secondary: 46B20, 47H09}
\keywords{Nonexpansive mapping; Asymptotically nonexpansive mapping; Fixed
point; Direct sum}
\address{Andrzej Wi\'{s}nicki, Institute of Mathematics, Maria Curie-Sk\l %
odowska University, 20-031 Lublin, Poland}
\email{awisnic@hektor.umcs.lublin.pl}
\date{October 29, 2012}
\maketitle

\section{Introduction}

A Banach space $X$ is said to have the fixed point property (FPP, for short)
if every nonexpansive mapping $T:C\rightarrow C,$ i.e.,
\begin{equation*}
\Vert Tx-Ty\Vert \leq \Vert x-y\Vert ,\ x,y\in C,
\end{equation*}%
acting on a (nonempty) bounded closed and convex subset $C$ of $X$ has a
fixed point. A Banach space $X$ is said to have the weak fixed point
property (WFPP, for short) if we additionally assume that $C$ is weakly
compact. The fixed point theorem of W. A. Kirk \cite{Ki} asserts that every
Banach space with weak normal structure has WFPP. Recall that a Banach space
$X$ has weak normal structure if $r(C)<\diam C$ for all weakly compact
convex subsets $C$ of $X$ consisting of more than one point, where $%
r(C)=\inf_{x\in C}\sup_{x\in C}\Vert x-y\Vert $ is the Chebyshev radius of $C
$. For more information regarding metric fixed point theory we refer the
reader to \cite{AyDoLo, GoKi, KiSi}.

The permanence properties of normal structure and other conditions which
guarantee the fixed point property under the direct sum operation has been
studied extensively since the 1968 Belluce--Kirk--Steiner theorem \cite{BKS}%
, which states that a direct sum of two Banach spaces with normal structure,
endowed with the maximum norm, also has normal structure. Nowadays, there
exist many results concerning permanence properties of conditions which
imply normal structure (see \cite{DhSa} for a survey). However, the problem
is more difficult if at least one of spaces lacks weak normal structure (see
\cite{Wi2} and references therein) and it is quite well understood only for
nonexpansive mappings defined on rectangles \thinspace $C_{1}\times C_{2}$
(see \cite{KiMa, Ku}).

It was recently proved in \cite{Wi2} that if a Banach space $X$ has WFPP and
$Y$ has the generalized Gossez-Lami Dozo property (which is a slightly
stronger property than weak normal structure) or is uniformly convex in
every direction, then the direct sum $X\oplus Y$ with a strictly monotone
norm has WFPP. The present paper is concerned with direct sums $(X_{1}\oplus
...\oplus X_{r})_{\psi }$ of Banach spaces with $M(X_{i})>1$ (see Section 2
for the definitions of a $\psi $-direct sum, $M(X)$ and $R(a,X)$).
Dhompongsa, Kaewcharoen and Kaewkhao \cite{DKK} proved that $(X_{1}\oplus
...\oplus X_{r})_{\psi }$ has WFPP whenever $M(X_{i})>1$ for each $i=1,...,r$
and $\psi \in \Psi _{r}$ is strictly convex. Kato and Tamura \cite{KaTa1}
proved that if $\psi \neq \psi _{1}$ then $R(X_{1}\oplus _{\psi }X_{2})<2$
iff $R(X_{1})<2$ and $R(X_{2})<2$. Subsequently, Kato and Tamura \cite{KaTa2}
showed that $R(a,(X_{1}\oplus ...\oplus X_{r})_{\infty })=\max_{1\leq i\leq
r}R(a,X_{i})$ and, consequently, $(X_{1}\oplus ...\oplus X_{r})_{\infty }$
has WFPP whenever $M(X_{i})>1$ for each $i$.

In this paper we show that $(X_{1}\oplus ...\oplus X_{r})_{\psi }$ has WFPP
if $M(X_{i})>1$ for each $i=1,...,r$ and the norm $\Vert \cdot \Vert _{\psi }
$ is strictly monotone. In paricular, $(X_{1}\oplus ...\oplus X_{r})_{\psi }$
has FPP if Banach spaces $X_{i}$ are uniformly nonsquare and $\Vert \cdot
\Vert _{\psi }$ is strictly monotone. This combined with the aforementioned
results gives a definitive answer for $r=2$: the direct sum\ $X_{1}\oplus
_{\psi }X_{2}$ of uniformly nonsquare spaces with any monotone norm has the
fixed point property. Theorems \ref{asympt} and \ref{asympt1} extend our
results for asymptotically nonexpansive mappings in the intermediate sense.

\section{Preliminaries}

The modulus $R(a,X)$ of a Banach space $X$ was defined by Dom\'{\i}nguez
Benavides \cite{Do1} as a generalization of the coefficient $R(X)$
introduced by Garc\'{\i}a Falset \cite{Ga1}. Recall that, for a given $a\geq
0,$%
\begin{equation*}
R(a,X)=\sup \{\liminf_{n\rightarrow \infty }\left\Vert x_{n}+x\right\Vert \},
\end{equation*}%
where the supremum is taken over all $x\in X$ with $\left\Vert x\right\Vert
\leq a$ and all weakly null sequences in the unit ball $B_{X}$ such that
\begin{equation*}
D[(x_{n})]=\limsup_{n\rightarrow \infty }\limsup_{m\rightarrow \infty
}\left\Vert x_{n}-x_{m}\right\Vert \leq 1.
\end{equation*}%
Define
\begin{equation*}
M(X)=\sup \left\{ \frac{1+a}{R(a,X)}:a\geq 0\right\} .
\end{equation*}%
Then the condition $M(X)>1$ implies that $X$ has the weak fixed point
property for nonexpansive mappings (see \cite{Do1}). We shall need the
following characterisation of $M(X)$ proved in \cite[Lemma 4.3]{BeWi} (see
also \cite[Corollary 4.3]{GLM}).

\begin{lemma}
\label{2.1} Let $X$ be a Banach space. The following conditions are
equivalent:

\begin{enumerate}
\item[(a)] $M(X)>1,$\smallskip

\item[(b)] there exists $a>0$ such that $R(a,X)<1+a,$\smallskip

\item[(c)] for every $a>0,$ $R(a,X)<1+a.$
\end{enumerate}
\end{lemma}

Let us now recall terminology concerning direct sums. A norm $\left\Vert
\cdot \right\Vert $ on $\mathbb{R}^{n}$ is said to be monotone if
\begin{equation*}
\Vert (x_{1},...,x_{n})\Vert \leq \Vert (y_{1},...,y_{n})\Vert
\end{equation*}%
whenever $\left\vert x_{1}\right\vert \leq \left\vert y_{1}\right\vert
,...,\left\vert x_{n}\right\vert \leq \left\vert y_{n}\right\vert $. A norm $%
\left\Vert \cdot \right\Vert $ is said to be strictly monotone if
\begin{equation*}
\Vert (x_{1},...,x_{n})\Vert <\Vert (y_{1},...,y_{n})\Vert
\end{equation*}%
whenever $\left\vert x_{i}\right\vert \leq \left\vert y_{i}\right\vert $ for
$i=1,...,n$ and $\left\vert x_{i_{0}}\right\vert <\left\vert
y_{i_{0}}\right\vert $ for some $i_{0}.$ It is easy to see that $\ell
_{p}^{n}$-norms, $1\leq p<\infty ,$ are strictly monotone. A norm $%
\left\Vert \cdot \right\Vert $ on $\mathbb{R}^{n}$ is said to be absolute if
\begin{equation*}
\Vert (x_{1},...,x_{n})\Vert =\Vert (\left\vert x_{1}\right\vert
,...,\left\vert x_{n}\right\vert )\Vert
\end{equation*}%
for every $(x_{1},x_{2},...,x_{n})\in \mathbb{R}^{n}.$ It is well known (see
\cite{BaStWi}) that a norm is absolute iff it is monotone.

We will assume that the norm is normalized, i.e.,
\begin{equation*}
\Vert (1,0,...,0)\Vert =...=\Vert (0,...,0,1)\Vert =1.
\end{equation*}%
Bonsal and Duncan \cite{BoDu} showed that the set of all absolute and
normalized norms on $\mathbb{R}^{2}$ $(\mathbb{C}^{2})$ is in one-to-one
correspondence with the set $\Psi $ of all continuous convex functions on $%
[0,1]$ satisfying $\psi (0)=\psi (1)=1$ and $\max \{1-t,t\}\leq \psi (t)\leq
1$ for $0\leq t\leq 1,$ where the correspondence is given by
\begin{equation}
\psi (t)=\left\Vert (1-t,t)\right\Vert ,\ 0\leq t\leq 1.  \label{n1}
\end{equation}%
Conversely, for any $\psi \in \Psi $ define%
\begin{equation*}
\Vert (x_{1},x_{2})\Vert _{\psi }=(\left\vert x_{1}\right\vert +\left\vert
x_{2}\right\vert )\psi (\left\vert x_{2}\right\vert /\left\vert
x_{1}\right\vert +\left\vert x_{2}\right\vert )
\end{equation*}%
for $(x_{1},x_{2})\neq (0,0)$ and $\Vert (0,0)\Vert _{\psi }=0.$ Then $\Vert
\cdot \Vert _{\psi }$ is an absolute and normalized norm which satisfies (%
\ref{n1}) (see \cite{BoDu, SaKaTa}). For example, the $\ell _{p}^{2}$-norms
correspond to the functions%
\begin{equation*}
\psi _{p}(t)=\left\{
\begin{array}{ll}
((1-t)^{p}+t^{p})^{1/p}\ \  & \text{if }1\leq p<\infty , \\
\max \{1-t,t\} & \text{if }p=\infty .%
\end{array}%
\right.
\end{equation*}%
It was proved in \cite[Corollary 3]{TaKaSa2} that a norm $\Vert \cdot \Vert
_{\psi }$ in $\mathbb{R}^{2}$ is normalized and strictly monotone iff
\begin{equation*}
\psi (t)>\psi _{\infty }(t)
\end{equation*}%
for all $0<t<1.$

Saito, Kato and Takahashi \cite{SaKaTa2} generalized the result of \cite%
{BoDu} to the $n$-dimensional case. Let%
\begin{equation*}
\triangle _{n}=\{(s_{1},...,s_{n-1})\in \mathbb{R}^{n-1}:s_{1}+...+s_{n-1}%
\leq 1,s_{i}\geq 0,i=1,...,n-1\}
\end{equation*}%
and denote by $\Psi _{n}$ the set of all continuous convex functions on $%
\triangle _{n}$ which satisfy the following conditions:

\begin{align*}
\psi (1,0,...,0) & = \psi (0,1,0,...,0)=...=\psi (0,...,0,1)=1, \\
\psi (s_{1},...,s_{n-1}) & \geq (s_{1}+...+s_{n-1})\psi (\frac{s_{1}}{%
s_{1}+...+s_{n-1}},...,\frac{s_{n-1}}{s_{1}+...+s_{n-1}}), \\
\psi (s_{1},...,s_{n-1}) & \geq (1-s_{1})\psi (0,\frac{s_{2}}{1-s_{1}},...,%
\frac{s_{n-1}}{1-s_{1}}), \\
\vdots \\
\psi (s_{1},...,s_{n-1}) & \geq (1-s_{n-1})\psi (\frac{s_{1}}{1-s_{n-1}},...,%
\frac{s_{n-2}}{1-s_{n-1}},0).
\end{align*}

Then the set of all absolute and normalized norms on $\mathbb{R}^{n}$ $(%
\mathbb{C}^{n})$ is in one-to-one correspondence with the set $\Psi _{n},$
where the correspondence is given by%
\begin{equation}
\psi (s_{1},...,s_{n-1})=\left\Vert
(1-\sum\nolimits_{i=1}^{n-1}s_{i},s_{1},...,s_{n-1})\right\Vert ,\
(s_{1},...,s_{n-1})\in \triangle _{n},  \label{n2}
\end{equation}%
and if $\psi \in \Psi _{n}$ then%
\begin{align*}
\Vert (x_{1},x_{2},...,x_{n})\Vert _{\psi } = & \ (\left\vert
x_{1}\right\vert +\left\vert x_{2}\right\vert +...+\left\vert
x_{n}\right\vert )\times \\
& \psi (\frac{\left\vert x_{2}\right\vert }{\left\vert x_{1}\right\vert
+\left\vert x_{2}\right\vert +...+\left\vert x_{n}\right\vert },...,\frac{%
\left\vert x_{n}\right\vert }{\left\vert x_{1}\right\vert +\left\vert
x_{2}\right\vert +...+\left\vert x_{n}\right\vert }),
\end{align*}
defined for $(x_{1},x_{2},...,x_{n})\neq (0,0,...,0),\Vert (0,0,...,0)\Vert
_{\psi }=0,$ is an absolute and normalized norm which satisfies (\ref{n2})
(see \cite[Theorem 3.4]{SaKaTa2}).

Let $X_{1},...,X_{r}$ be Banach spaces and $\psi \in \Psi _{r}$. Following
\cite{TaKaSa2}, we shall write $(X_{1}\oplus ...\oplus X_{r})_{\psi }$ for
the $\psi $-direct sum with the norm $\Vert (x_{1},...,x_{r})\Vert _{\psi
}=\Vert (\Vert x_{1}\Vert ,...,\Vert x_{r}\Vert )\Vert _{\psi }$, where $%
(x_{1},...,x_{r})\in X_{1}\times ...\times X_{r}$.

\section{Fixed point theorems}

Let $T:C\rightarrow C$ be a nonexpansive mapping, where $C$ is a nonempty
weakly compact convex subset of a Banach space $X$. By the Kuratowski--Zorn
lemma, there exists a minimal (in the sense of inclusion) convex and weakly
compact set $K\subset C$ which is invariant under $T.$ Let $(x_{n})$ be an
approximate fixed point sequence for $T$ in $K$, i.e., $\lim_{n\rightarrow
\infty }\Vert Tx_{n}-x_{n}\Vert =0.$ The Goebel-Karlovitz lemma (see \cite%
{Go, Ka}) asserts that
\begin{equation*}
\lim_{n\rightarrow \infty }\Vert x_{n}-x\Vert =\diam K
\end{equation*}%
for every $x\in K$.

Suppose now that $X_{1},X_{2},...,X_{r}$ are Banach spaces, $\psi \in \Psi
_{r}$ and $T:K\rightarrow K$ is a nonexpansive mapping acting on a weakly
compact convex and minimal invariant subset $K$ of a direct sum $%
(X_{1}\oplus ...\oplus X_{r})_{\psi }.$ Suppose further that $\diam K=1$ and
$(w_{n})=((x_{n}^{(1)},...,x_{n}^{(r)}))$ is an approximate fixed point
sequence for $T$ in $K$ weakly converging to $(0,...,0)\in K$ such that $%
\lim_{n\rightarrow \infty }\Vert x_{n}^{(r)}\Vert =0.$ The following
construction (for $r=2$) was proposed in \cite{Wi2}. Fix an integer $k\geq 1$
and a sequence $(\varepsilon _{n})$ in $(0,1).$ Since $\Vert
Tw_{n}-w_{n}\Vert _{\psi }$ and $\Vert x_{n}^{(r)}\Vert $ converge to $0,$
we can choose $w_{n_{1}}$ in such a way that $\Vert
Tw_{n_{1}}-w_{n_{1}}\Vert _{\psi }<\varepsilon _{1}$ and $\Vert
x_{n_{1}}^{(r)}\Vert <\varepsilon _{1}.$ Let us put
\begin{equation*}
D_{1}^{1}=\left\{ w_{n_{1}}\right\} ,D_{2}^{1}=\conv(\{w_{n_{1}},Tw_{n_{1}}%
\}),...,D_{k}^{1}=\conv(D_{k-1}^{1}\cup T(D_{k-1}^{1})).
\end{equation*}%
Thus we obtain a family $D_{1}^{1}\subset D_{2}^{1}\subset ...\subset
D_{k}^{1}$ of relatively compact convex subsets of $K.$ It follows from the
Goebel-Karlovitz lemma and the relative compactness of $D_{k}^{1}$ that
there exists $n_{2}>n_{1}$ such that $\Vert Tw_{n_{2}}-w_{n_{2}}\Vert _{\psi
}<\varepsilon _{2}$, $\Vert x_{n_{2}}^{(r)}\Vert <\varepsilon _{2}$ and $%
\Vert w_{n_{2}}-z\Vert _{\psi }>1-\varepsilon _{2}$ for all $z\in D_{k}^{1}.$
Put%
\begin{equation*}
D_{1}^{2}=\conv(\{w_{n_{1}},w_{n_{2}}\}),...,D_{k}^{2}=\conv(D_{k-1}^{2}\cup
T(D_{k-1}^{2})).
\end{equation*}%
Again, we can find $n_{3}>n_{2}$ such that $\Vert Tw_{n_{3}}-w_{n_{3}}\Vert
_{\psi }<\varepsilon _{3}$, $\Vert x_{n_{3}}^{(r)}\Vert <\varepsilon _{3}$
and $\Vert w_{n_{3}}-z\Vert _{\psi }>1-\varepsilon _{3}$ for all $z\in
D_{k}^{2}.$ Put%
\begin{equation*}
D_{1}^{3}=\conv(\{w_{n_{1}},w_{n_{2}},w_{n_{3}}\}),...,D_{k}^{3}=\conv%
(D_{k-1}^{3}\cup T(D_{k-1}^{3})).
\end{equation*}%
Continuing in this fashion, we obtain by induction a subsequence $%
(w_{n_{j}}) $ of $(w_{n})$ and a family $\left\{ D_{j}^{i}\right\} _{1\leq
j\leq k,i\geq 1}$ of subsets of $K$ such that

\begin{enumerate}
\item[(i)] $\Vert Tw_{n_{i}}-w_{n_{i}}\Vert _{\psi }<\varepsilon _{i},$

\item[(ii)] $\Vert x_{n_{i}}^{(r)}\Vert <\varepsilon _{i},$

\item[(iii)] $\Vert w_{n_{i}}-z\Vert _{\psi }>1-\varepsilon _{i}$ for all $%
z\in D_{k}^{i-1},$ (where $D_{k}^{0}=\emptyset $),

\item[(iv)] $D_{1}^{i}=\conv(\left\{
w_{n_{1}},w_{n_{2}},...,w_{n_{i}}\right\} ),$

\item[(v)] $D_{j+1}^{i}=\conv(D_{j}^{i}\cup T(D_{j}^{i})),$
\end{enumerate}

for every $i\geq 1$ and $j=1,...,k-1$ (see \cite[Lemma 3.1]{Wi2}).

Furthermore (see \cite[Lemma 3.2]{Wi2}), for every $i\geq 1,$ $j=1,...,k$
and $u\in D_{j}^{i+1}$, there exists $z\in D_{j}^{i}$ such that
\begin{equation*}
\left\Vert z-u\right\Vert _{\psi }+\left\Vert u-w_{n_{i+1}}\right\Vert
_{\psi }\leq \left\Vert z-w_{n_{i+1}}\right\Vert _{\psi }+3(j-1)\varepsilon
_{i+1}.
\end{equation*}

We need the following simple observation regarding the set $D_{k}^{1}.$

\begin{lemma}
\label{pom}If $u=(y^{(1)},...,y^{(r)})\in D_{k}^{1}$ then $\left\Vert
y^{(r)}\right\Vert <k\varepsilon _{1}.$
\end{lemma}

\begin{proof}
We first show by induction that for every $j\in \{2,...,k\}$ and for every $%
v\in D_{j}^{1},$%
\begin{equation}
\left\Vert v-w_{n_{1}}\right\Vert _{\psi }<(j-1)\varepsilon _{1}.
\label{pom1}
\end{equation}%
For $j=2,$ $D_{2}^{1}=\conv(\{w_{n_{1}},Tw_{n_{1}}\})$ and the inequality (%
\ref{pom1}) follows from the fact that $\Vert Tw_{n_{1}}-w_{n_{1}}\Vert
_{\psi }<\varepsilon _{1}.$

Now fix $j\in \{2,...,k-1\}$ and assume that for every $u\in D_{j}^{1},$ $%
\left\Vert u-w_{n_{1}}\right\Vert _{\psi }<(j-1)\varepsilon _{1}.$ Since%
\begin{equation*}
D_{j+1}^{1}=\conv(D_{j}^{1}\cup T(D_{j}^{1})),
\end{equation*}%
it is enough to show that $\left\Vert v-w_{n_{1}}\right\Vert _{\psi
}<j\varepsilon _{1}$for every $v\in T(D_{j}^{1}).$ Let $v=Tu$ for some $u\in
D_{j}^{1}.$ Then%
\begin{equation*}
\left\Vert v-w_{n_{1}}\right\Vert _{\psi }\leq \left\Vert
Tu-Tw_{n_{1}}\right\Vert _{\psi }+\left\Vert Tw_{n_{1}}-w_{n_{1}}\right\Vert
_{\psi }\leq \left\Vert u-w_{n_{1}}\right\Vert _{\psi }+\varepsilon
_{1}<j\varepsilon _{1},
\end{equation*}%
and the proof of (\ref{pom1}) is complete.

Let $w_{n_{1}}=(x_{n_{1}}^{(1)},...,x_{n_{1}}^{(r)}),u=(y^{(1)},...,y^{(r)})%
\in D_{k}^{1}.$ It follows from (\ref{pom1}) that%
\begin{equation*}
\left\Vert y^{(r)}-x_{n_{1}}^{(r)}\right\Vert \leq \left\Vert
u-w_{n_{1}}\right\Vert _{\psi }<(k-1)\varepsilon _{1}
\end{equation*}%
and consequently, since $\left\Vert x_{n_{1}}^{(r)}\right\Vert <\varepsilon
_{1}$, $\left\Vert y^{(r)}\right\Vert <k\varepsilon _{1}.$
\end{proof}

Notice that if the norm $\Vert \cdot \Vert _{\psi }$ on $\mathbb{R}^{r}$ is
strictly monotone then for every $\varepsilon >0,$ there exists $\delta
(\varepsilon )>0$ such that if $(t^{(1)},...,t^{(r-1)},t^{\prime
}),(t^{(1)},...,t^{(r-1)},t^{\prime \prime })$ belong to the unit ball of $(%
\mathbb{R}^{r},\Vert \cdot \Vert _{\psi })$ and
\begin{equation*}
\left\Vert (t^{(1)},...,t^{(r-1)},t^{\prime })\right\Vert _{\psi
}<\left\Vert (t^{(1)},...,t^{(r-1)},t^{\prime \prime })\right\Vert _{\psi
}+\delta (\varepsilon ),
\end{equation*}%
then $\left\vert t^{\prime }\right\vert <\left\vert t^{\prime \prime
}\right\vert +\varepsilon .$ Indeed, otherwise, there exist $\varepsilon
_{0}>0$ and two sequences $((t_{n}^{(1)},...,t_{n}^{(r-1)},t_{n}^{\prime
})),((t_{n}^{(1)},...,t_{n}^{(r-1)},t_{n}^{\prime \prime }))$ in the unit
ball such that
\begin{equation*}
\left\Vert (t_{n}^{(1)},...,t_{n}^{(r-1)},t_{n}^{\prime })\right\Vert _{\psi
}<\left\Vert (t_{n}^{(1)},...,t_{n}^{(r-1)},t_{n}^{\prime \prime
})\right\Vert _{\psi }+\frac{1}{n}
\end{equation*}%
and $\left\vert t_{n}^{\prime }\right\vert \geq \left\vert t_{n}^{\prime
\prime }\right\vert +\varepsilon _{0}.$ Passing to convergent subsequences
and taking limits, we have
\begin{equation*}
\left\Vert (g^{(1)},...,g^{(r-1)},g^{\prime })\right\Vert _{\psi }\leq
\left\Vert (g^{(1)},...,g^{(r-1)},g^{\prime \prime })\right\Vert _{\psi }
\end{equation*}%
and $\left\vert g^{\prime }\right\vert \geq \left\vert g^{\prime \prime
}\right\vert +\varepsilon _{0}$ which contradicts the strict monotonicity of
$\Vert \cdot \Vert _{\psi }.$

Thus, using Lemma \ref{pom}, we can follow the arguments in \cite[Lemma 3.3]%
{Wi2} and obtain the following result.

\begin{lemma}
\label{close2}Let $X_{1},X_{2},...,X_{r}$ be Banach spaces, $\psi \in \Psi
_{r},$ and $T:K\rightarrow K$ a nonexpansive mapping acting on a weakly
compact convex and minimal invariant subset $K$ of a direct sum $%
(X_{1}\oplus ...\oplus X_{r})_{\psi }$ with a strictly monotone norm $\Vert
\cdot \Vert _{\psi }$. Suppose that $\diam K=1$ and $%
(w_{n})=((x_{n}^{(1)},...,x_{n}^{(r)}))$ is an approximate fixed point
sequence for $T$ in $K$ weakly converging to $(0,...,0)\in K$ such that $%
\lim_{n\rightarrow \infty }\Vert x_{n}^{(r)}\Vert =0.$ Then for every
positive integer $k$, there exist a sequence $(\varepsilon _{n})$ in $(0,1),$
a subsequence $(w_{n_{j}})$ of $(w_{n})$ and a family $\left\{
D_{j}^{i}\right\} _{1\leq j\leq k,i\geq 1}$ of subsets of $K$ (depending on $%
k$) such that the above conditions (i)-(v) are satisfied and $\left\Vert
y^{(r)}\right\Vert <\frac{1}{k}$ for every $u=(y^{(1)},...,y^{(r)})\in
\bigcup\nolimits_{i\mathbb{=}1}^{\infty }D_{k}^{i}.$
\end{lemma}

We show that the appropriate approximate fixed point sequence $(w_{n})$
exists if $M(X_{i})>1,i=1,...,r.$

\begin{lemma}
\label{lem}Let $X_{1},...,X_{r}$ be Banach spaces with $M(X_{i})>1$ for each
$i=1,...,r$, $\psi \in \Psi _{r},$ and $T:K\rightarrow K$ a nonexpansive
mapping acting on a weakly compact convex and minimal invariant subset $K$
of a direct sum $(X_{1}\oplus ...\oplus X_{r})_{\psi }$ with a strictly
monotone norm $\Vert \cdot \Vert _{\psi }.$ Suppose that $\diam K=1$ and $%
(w_{n})=((x_{n}^{(1)},...,x_{n}^{(r)}))$ is an approximate fixed point
sequence for $T$ in $K$ weakly converging to $(0,...,0)\in K.$ Then $%
\liminf_{n\rightarrow \infty }\Vert x_{n}^{(i_{0})}\Vert =0$ for some $%
i_{0}\in \{1,...,r\}.$
\end{lemma}

\begin{proof}
Let $(w_{n})=((x_{n}^{(1)},...,x_{n}^{(r)}))$ be an approximate fixed point
sequence for $T$ in $K$ weakly converging to $(0,...,0)\in K$ and suppose,
contrary to our claim, that for each $i=1,...,r,$ $\liminf_{n\rightarrow
\infty }\Vert x_{n}^{(i)}\Vert >0.$ We can assume, passing to a subsequence,
that
\begin{equation*}
d_{1}=\min \{\lim_{n\rightarrow \infty }\Vert x_{n}^{(1)}\Vert
,...,\lim_{n\rightarrow \infty }\Vert x_{n}^{(r)}\Vert \}>0.
\end{equation*}

We claim that for every $\varepsilon >0$ there exists $\delta
_{1}(\varepsilon )>0$ such that if $v\in K$ and $\Vert Tv-v\Vert _{\psi
}<\delta _{1}(\varepsilon )$ then $\Vert v\Vert _{\psi }>1-\varepsilon $.
Indeed, otherwise, arguing as in \cite[Lemma 2]{Do3}, there exist $%
\varepsilon _{0}>0$ and a sequence $(v_{n})\subset K$ such that $\Vert
Tv_{n}-v_{n}\Vert _{\psi }<\frac{1}{n}$ and $\Vert v_{n}\Vert _{\psi }\leq
1-\varepsilon _{0}$ for every $n\in \mathbb{N}$. Thus $(v_{n})$ is an
approximate fixed point sequence in $K$, but $\limsup_{n\rightarrow \infty
}\Vert v_{n}\Vert _{\psi }\leq 1-\varepsilon _{0},$ which contradicts the
Goebel-Karlovitz lemma since $(0,...,0)\in K.$

Furthermore, since $(\mathbb{R}^{r},\left\Vert \cdot \right\Vert _{\psi })$
is finite dimensional and the norm $\left\Vert \cdot \right\Vert _{\psi }$
is strictly monotone, for every $\eta >0,$ there exists $\delta _{2}(\eta
)>0 $ such that if $(t^{(1)},...,t^{(r)}),(s^{(1)},...,s^{(r)})$ belong to
the unit ball, $\left\vert t^{(i)}\right\vert \leq \left\vert
s^{(i)}\right\vert $ for $i=1,...,r$ and $\left\vert t^{(i_{0})}\right\vert
<\left\vert s^{(i_{0})}\right\vert -\eta $ for some $i_{0},$ then
\begin{equation*}
\left\Vert (t^{(1)},...,t^{(r)})\right\Vert _{\psi }<\left\Vert
(s^{(1)},...,s^{(r)})\right\Vert _{\psi }-\delta _{2}(\eta ).
\end{equation*}

Let
\begin{equation*}
A=\inf_{d_{1}/6r\leq a\leq 1}\min_{1\leq i\leq r}(1+a-R(a,X_{i}))
\end{equation*}%
and take $0<\eta <Ad_{1}/3,\varepsilon <\min \{\delta _{2}(\eta ),d_{1}/6\}$
(notice that $A>0$ by Lemma \ref{2.1}). Choose $t=1-\frac{1}{3}d_{1},\gamma
<\min \{1,\delta _{1}(\varepsilon )\}$ and define the contraction $%
S_{n}:K\rightarrow K$ by%
\begin{equation*}
S_{n}x=(1-\gamma )Tx+\gamma tw_{n}.
\end{equation*}%
By the contractive mapping principle, for any $n\in \mathbb{N},$ there
exists a unique fixed point $z_{n}=(y_{n}^{(1)},...,y_{n}^{(r)})$ of $S_{n}.$
We can assume, passing to a subsequence, that $(z_{n})$ converges weakly to $%
z=(y^{(1)},...,y^{(r)})$ and the limits $\lim_{n\rightarrow \infty
}\left\Vert y_{n}^{(i)}\right\Vert ,\lim_{n,m\rightarrow \infty ,n\neq
m}\left\Vert y_{n}^{(i)}-y_{m}^{(i)}\right\Vert $ exist for each $i=1,...,r$
(see, e.g., \cite{SiSm}).

It is not difficult to see that%
\begin{equation}
\left\Vert Tz_{n}-z_{n}\right\Vert _{\psi }\leq \gamma <\delta
_{1}(\varepsilon )  \label{w1}
\end{equation}%
and hence%
\begin{equation}
\left\Vert z_{n}\right\Vert _{\psi }>1-\varepsilon   \label{eps}
\end{equation}%
for every $n\in \mathbb{N}$. Furthermore, for every $n,m\in \mathbb{N},$%
\begin{equation}
\left\Vert z_{n}-z_{m}\right\Vert _{\psi }\leq t  \label{w2}
\end{equation}%
and from the weak lower semicontinuity of the norm,%
\begin{equation*}
\left\Vert z_{n}-z\right\Vert _{\psi }\leq \liminf_{m\rightarrow \infty
}\left\Vert z_{n}-z_{m}\right\Vert _{\psi }\leq t.
\end{equation*}%
Thus%
\begin{equation*}
\left\Vert z\right\Vert _{\psi }\geq \left\Vert z_{n}\right\Vert _{\psi
}-\left\Vert z_{n}-z\right\Vert _{\psi }>1-\varepsilon -t>d_{1}/6
\end{equation*}%
and, consequently, there exists $i_{0}\in \{1,...,r\}$ such that
\begin{equation}
\left\Vert y^{(i_{0})}\right\Vert >d_{1}/6r  \label{odzera}
\end{equation}%
since, otherwise,%
\begin{equation*}
\left\Vert z\right\Vert _{\psi }\leq \left\Vert y^{(1)}\right\Vert
+\left\Vert y^{(2)}\right\Vert +...+\left\Vert y^{(r)}\right\Vert \leq
d_{1}/6.
\end{equation*}%
Furthermore,%
\begin{equation*}
\left\Vert z_{n}-w_{n}\right\Vert _{\psi }\leq (1-\gamma )\left\Vert
Tz_{n}-Tw_{n}\right\Vert _{\psi }+(1-\gamma )\left\Vert
Tw_{n}-w_{n}\right\Vert _{\psi }+\gamma (1-t)\left\Vert w_{n}\right\Vert
_{\psi }
\end{equation*}%
which gives%
\begin{equation*}
\left\Vert z_{n}-w_{n}\right\Vert _{\psi }\leq 1-t+\frac{(1-\gamma )}{\gamma
}\left\Vert Tw_{n}-w_{n}\right\Vert _{\psi }.
\end{equation*}%
It follows from the weak lower semicontinuity of the norm that for each $%
i=1,...,r,$%
\begin{equation*}
\left\Vert y^{(i)}\right\Vert \leq \liminf_{n\rightarrow \infty }\left\Vert
y_{n}^{(i)}-x_{n}^{(i)}\right\Vert \leq \liminf_{n\rightarrow \infty
}\left\Vert z_{n}-w_{n}\right\Vert _{\psi }\leq 1-t=\frac{1}{3}d_{1}
\end{equation*}%
and from the triangle inequality,%
\begin{equation*}
\left\Vert y_{n}^{(i)}-y^{(i)}\right\Vert \geq \left\Vert
x_{n}^{(i)}\right\Vert -\left\Vert x_{n}^{(i)}-y_{n}^{(i)}\right\Vert
-\left\Vert y^{(i)}\right\Vert .
\end{equation*}%
Hence%
\begin{equation}
\lim_{n\rightarrow \infty }\left\Vert y_{n}^{(i)}-y^{(i)}\right\Vert \geq
\frac{1}{3}d_{1}  \label{w3}
\end{equation}%
for each $i=1,...,r.$ Write%
\begin{equation*}
d_{2}=\lim_{n,m\rightarrow \infty ,n\neq m}\left\Vert
y_{n}^{(i_{0})}-y_{m}^{(i_{0})}\right\Vert
\end{equation*}%
and notice that
\begin{equation*}
d_{2}=\limsup_{n\rightarrow \infty }\limsup_{m\rightarrow \infty }\left\Vert
y_{n}^{(i_{0})}-y_{m}^{(i_{0})}\right\Vert \geq \lim_{n\rightarrow \infty
}\left\Vert y_{n}^{(i_{0})}-y^{(i_{0})}\right\Vert \geq \frac{1}{3}d_{1}
\end{equation*}%
by (\ref{w3}). It follows that
\begin{align*}
\lim_{n\rightarrow \infty }\left\Vert y_{n}^{(i_{0})}\right\Vert &
=\lim_{n\rightarrow \infty }\left\Vert \frac{y_{n}^{(i_{0})}-y^{(i_{0})}}{%
d_{2}}+\frac{y^{(i_{0})}}{d_{2}}\right\Vert d_{2} \\
& \leq R(\frac{\left\Vert y^{(i_{0})}\right\Vert }{d_{2}},X_{i_{0}})d_{2}%
\leq (1+\frac{\left\Vert y^{(i_{0})}\right\Vert }{d_{2}}-A)d_{2} \\
& \leq d_{2}+\left\Vert y^{(i_{0})}\right\Vert -Ad_{2}<d_{2}+\left\Vert
y^{(i_{0})}\right\Vert -\eta ,
\end{align*}%
since $d_{1}/6r\leq \left\Vert y^{(i_{0})}\right\Vert /d_{2}\leq 1.$
Therefore,
\begin{align*}
\lim_{n\rightarrow \infty }\left\Vert z_{n}\right\Vert _{\psi }& <\left\Vert
(\lim_{n\rightarrow \infty }\left\Vert y_{n}^{(1)}\right\Vert
,...,d_{2}+\left\Vert y^{(i_{0})}\right\Vert ,...,\lim_{n\rightarrow \infty
}\left\Vert y_{n}^{(r)}\right\Vert )\right\Vert _{\psi }-\delta _{2}(\eta )
\\
& \leq \left\Vert (\lim_{n\rightarrow \infty }\left\Vert
y_{n}^{(1)}-y^{(1)}\right\Vert ,...,\lim_{n\rightarrow \infty }\left\Vert
y_{n}^{(r)}-y^{(r)}\right\Vert )\right\Vert _{\psi } \\
& +\left\Vert (\left\Vert y^{(1)}\right\Vert ,...,\left\Vert
y^{(r)}\right\Vert )\right\Vert _{\psi }-\varepsilon \leq
\lim_{n,m\rightarrow \infty ,n\neq m}\left\Vert z_{n}-z_{m}\right\Vert
_{\psi } \\
& +\liminf_{n\rightarrow \infty }\left\Vert z_{n}-w_{n}\right\Vert _{\psi
}-\varepsilon \leq t+1-t-\varepsilon \leq 1-\varepsilon
\end{align*}%
which contradicts (\ref{eps}).
\end{proof}

We can now formulate our main result. The proof combines the arguments of
\cite[Theorem 3.4]{Wi2} and Lemma \ref{lem}.

\begin{theorem}
\label{Th1}Let $X_{1},...,X_{r}$ be Banach spaces with $M(X_{i})>1$ for each
$i=1,...,r$, $\psi \in \Psi _{r}.$ Then the direct sum $(X_{1}\oplus
...\oplus X_{r})_{\psi }$ with a strictly monotone norm $\Vert \cdot \Vert
_{\psi }$ has the WFPP property.
\end{theorem}

\begin{proof}
Assume that $(X_{1}\oplus ...\oplus X_{r})_{\psi }$ does not have WFPP.
Then, there exist a weakly compact convex subset $C$ of $(X_{1}\oplus
...\oplus X_{r})_{\psi }$ and a nonexpansive mapping $T:C\rightarrow C$
without a fixed point. By the Kuratowski-Zorn lemma, there exists a convex
and weakly compact set $K\subset C$ which is minimal invariant under $T$ and
which is not a singleton. Let $(w_{n})=((x_{n}^{(1)},...,x_{n}^{(r)}))$ be
an approximate fixed point sequence for $T$ in $K.$ Without loss of
generality we can assume that $\diam K=1,$ $(w_{n})$ converges weakly to $%
(0,...,0)\in K$ and the limits $\lim_{n,m\rightarrow \infty ,n\neq m}\Vert
w_{n}-w_{m}\Vert _{\psi },$ $\lim_{n\rightarrow \infty }\Vert
x_{n}^{(i)}\Vert ,i=1,...,r,$ exist. Applying Lemma \ref{lem} gives $%
\lim_{n\rightarrow \infty }\Vert x_{n}^{(i_{0})}\Vert =0$ for some $i_{0}\in
\{1,...,r\}$ and by rearrangement of Banach spaces we can assume that $%
i_{0}=r.$ Now we follow the arguments in \cite{Wi2}. Lemma \ref{close2}
shows that for every positive integer $k,$ there exist a sequence $%
(\varepsilon _{n})$ in $(0,1),$ a subsequence $(w_{n_{j}})$ of $(w_{n})$ and
a family $\left\{ D_{j}^{i}\right\} _{1\leq j\leq k,i\geq 1}$ of subsets of $%
K$ (depending on $k$) such that the conditions (i)-(v) are satisfied and $%
\left\Vert y^{(r)}\right\Vert <\frac{1}{k}$ for every $%
u=(y^{(1)},...,y^{(r)})\in \bigcup\nolimits_{i\mathbb{=}1}^{\infty
}D_{k}^{i}.$

Let $C_{0}=\{(0,...,0)\}$ and $C_{j}=\conv(C_{j-1}\cup T(C_{j-1}))$ for $%
j\geq 1.$ It is not difficult to see that $\cl(\bigcup\nolimits_{j\mathbb{=}%
1}^{\infty }C_{j})$ is a closed convex subset of $K$ which is invariant for $%
T$ (and hence equals $K$). Fix $k\geq 1$ and notice that $(0,...,0)\in \cl%
(\bigcup\nolimits_{i\mathbb{=}1}^{\infty }D_{1}^{i})$, because the sequence $%
(w_{n_{j}})_{j\geq 1}$ converges weakly to $(0,...,0).$ Furthermore, for $%
j<k,$
\begin{equation*}
T(\cl(\bigcup\nolimits_{i\mathbb{=}1}^{\infty }D_{j}^{i}))=\cl%
(\bigcup\nolimits_{i\mathbb{=}1}^{\infty }T(D_{j}^{i}))\subset \cl%
(\bigcup\nolimits_{i\mathbb{=}1}^{\infty }D_{j+1}^{i})
\end{equation*}%
and hence, by induction on $j,$
\begin{equation*}
C_{j}\subset \cl(\bigcup\nolimits_{i\mathbb{=}1}^{\infty
}D_{j+1}^{i})\subset \cl(\bigcup\nolimits_{i\mathbb{=}1}^{\infty
}D_{k}^{i}),\ j<k.
\end{equation*}%
It follows that if $(y^{(1)},...,y^{(r)})\in C_{j}$ and $j<k,$ then $%
(y^{(1)},...,y^{(r)})\in \cl(\bigcup\nolimits_{i\mathbb{=}1}^{\infty
}D_{k}^{i})$ and consequently, $\left\Vert y^{(r)}\right\Vert \leq \frac{1}{k%
}.$ Letting $k\rightarrow \infty ,$ we have $y^{(r)}=0$ for every $%
(y^{(1)},...,y^{(r)})\in \cl(\bigcup\nolimits_{j\mathbb{=}1}^{\infty
}C_{j})=K.$ Therefore, $K$ is a subset of $(X_{1}\oplus ...\oplus
X_{r-1}\oplus \{0\})_{\psi }$ which is isometric to $(X_{1}\oplus ...\oplus
X_{r-1})_{\psi ^{\prime }}$ with the strictly monotone norm $\Vert
(x_{1},...,x_{r-1})\Vert _{\psi ^{\prime }}=\Vert (x_{1},...,x_{r-1},0)\Vert
_{\psi }.$ Repeating the above procedure $r-1$ times, we deduce that $K$ is
isometric to some $X_{j_{0}},j_{0}\in \{1,...,r\}.$ Since $M(X_{j_{0}})>1,$ $%
T$ has a fixed point in $K,$ which contradicts our assumption.
\end{proof}

Garc\'{\i}a Falset, Llor\'{e}ns Fuster and Mazcu\~{n}an Navarro \cite{GLM}
introduced modulus $RW(a,X),$ which plays an important role in fixed point
theory for nonexpansive mappings. For a given $a\geq 0,$%
\begin{equation*}
RW(a,X)=\sup \min \{\liminf_{n}\left\Vert x_{n}+x\right\Vert
,\liminf_{n}\left\Vert x_{n}-x\right\Vert \},
\end{equation*}%
where the supremum is taken over all $x\in X$ with $\left\Vert x\right\Vert
\leq a$ and all weakly null sequences in the unit ball $B_{X}.$ Let
\begin{equation*}
MW(X)=\sup \left\{ \frac{1+a}{RW(a,X)}:a\geq 0\right\} .
\end{equation*}%
It was proved in \cite[Theorem 3.3]{GLM} that $M(X)\geq MW(X)$ whenever $%
B_{X^{\ast }}$ is $w^{\ast }$-sequentially compact. Since a minimal
invariant set $K$ is always separable, we obtain the following corollary.

\begin{corollary}
Let $X_{1},...,X_{r}$ be Banach spaces with $MW(X_{i})>1$ for each $%
i=1,...,r $, $\psi \in \Psi _{r}.$ Then the direct sum $(X_{1}\oplus
...\oplus X_{r})_{\psi }$ with a strictly monotone norm $\Vert \cdot \Vert
_{\psi }$ has WFPP.
\end{corollary}

Recall that a Banach space $X$ is uniformly nonsquare if
\begin{equation*}
J(X)=\sup_{x,y\in S_{X}}\min \left\{ \left\Vert x+y\right\Vert ,\left\Vert
x-y\right\Vert \right\} <2.
\end{equation*}%
Corollary 4.4 in \cite{GLM} shows that if $X$ is uniformly nonsquare, then $%
MW(X)>1.$ Since uniformly nonsquare spaces are reflexive, properties FPP and
WFPP coincide and thus we obtain the following corollary.

\begin{corollary}
\label{cor2}Let $X_{1},...,X_{r}$ be uniformly nonsquare Banach spaces, $%
\psi \in \Psi _{r}$. Then the direct sum $(X_{1}\oplus ...\oplus
X_{r})_{\psi }$ with a strictly monotone norm $\Vert \cdot \Vert _{\psi }$
has FPP.
\end{corollary}

In the case $r=2$ we have a stronger, definitive result.

\begin{theorem}
\label{th2}Let $X_{1},X_{2}$ be uniformly nonsquare Banach spaces, $\psi \in
\Psi _{2}$. Then the direct sum $X_{1}\oplus _{\psi }X_{2}$ has FPP.
\end{theorem}

\begin{proof}
Kato, Sato and Tamura \cite[Theorem 1]{KaSaTa} proved that $X_{1}\oplus
_{\psi }X_{2}$ is uniformly nonsquare iff $X_{1},X_{2}$ are uniformly
nonsquare and $\psi \neq \psi _{1},\psi _{\infty }$ (see also \cite{BePr}).
Kato and Tamura \cite[Theorem 3.6]{KaTa1} showed that $R(X_{1}\oplus
_{\infty }X_{2})<2$ and hence $X_{1}\oplus _{\infty }X_{2}$ has FPP. The
case $X_{1}\oplus _{1}X_{2}$ is covered by Corollary \ref{cor2}, since the
norm $\Vert \cdot \Vert _{\psi _{1}}$ is strictly monotone.
\end{proof}

We conclude with two fixed point theorems for asymptotically nonexpansive
mappings. Recall that a mapping $T:C\rightarrow C$ is said to be
asymptotically nonexpansive in the intermediate sense if $T$ is continuous
and%
\begin{equation*}
\limsup_{n\rightarrow \infty }\sup_{x,y\in C}(\Vert T^{n}x-T^{n}y\Vert
-\Vert x-y\Vert )\leq 0.
\end{equation*}%
A Banach space $X$ is said to have the super fixed point property for
nonexpansive mappings (resp. asymptotically nonexpansive mappings) if every
Banach space $Y$ which is finitely representable in $X$ has the fixed point
property for nonexpansive mappings (resp. asymptotically nonexpansive
mappings).

Theorem 2.4 in \cite{Wi3} shows that $X$ has the super fixed point property
for nonexpansive mappings if and only if $X$ has the super fixed point
property for asymptotically nonexpansive mappings in the intermediate sense.
Since the direct sum $(X_{1}\oplus ...\oplus X_{r})_{\psi }$ of uniformly
nonsquare spaces is stable under passing to the Banach space ultrapowers, it
follows from the properties of ultrapowers and Corollary \ref{cor2} (resp.
Theorem \ref{th2}) that $(X_{1}\oplus ...\oplus X_{r})_{\psi }$ with a
strictly monotone norm (resp. $X_{1}\oplus _{\psi }X_{2}$ with any monotone
norm) has the super fixed property for nonexpansive mappings and,
consequently, for asymptotically nonexpansive mappings in the intermediate
sense. Thus we obtain the following theorem.

\begin{theorem}
\label{asympt}Let $X_{1},...,X_{r}$ be uniformly nonsquare Banach spaces, $%
\psi \in \Psi _{r}$. Then the direct sum $(X_{1}\oplus ...\oplus
X_{r})_{\psi }$ with a strictly monotone norm $\Vert \cdot \Vert _{\psi }$
has the (super) fixed point property for asymptotically nonexpansive
mappings in the intermediate sense. The assumption about the strict
monotonicity of the norm can be dropped if $r=2$.
\end{theorem}

Recall that Theorem 2.3 in \cite{Wi1} shows that the direct sum $%
(X_{1}\oplus ...\oplus X_{r})_{\psi }$ of uniformly noncreasy spaces with a
strictly monotone norm has FPP. Since uniformly noncreasy spaces are stable
under passing to the Banach space ultrapowers (see \cite{Pr}), the
conclusion of Theorem \ref{asympt} is also valid in this case.

\begin{theorem}
\label{asympt1}Let $X_{1},...,X_{r}$ be uniformly noncreasy Banach spaces, $%
\psi \in \Psi _{r}$. Then the direct sum $(X_{1}\oplus ...\oplus
X_{r})_{\psi }$ with a strictly monotone norm $\Vert \cdot \Vert _{\psi }$
has the super fixed point property for asymptotically nonexpansive mappings
in the intermediate sense.
\end{theorem}

\begin{acknowledgement}
The author wishes to express his thanks to Professor Tom\'{a}s Dom\'{\i}%
nguez Benavides for fruitful conversations regarding the subject of this
paper.
\end{acknowledgement}

\end{document}